\documentclass[11pt]{amsart}
\usepackage{mathtools} \mathtoolsset{showonlyrefs,showmanualtags}
\usepackage{hyperref} 
\hypersetup{
    colorlinks=true,       
    linkcolor=blue,          
    citecolor=magenta,        
    filecolor=magenta,      
    urlcolor=cyan           
}

\usepackage{xcolor}
\numberwithin{equation}{section}
\allowdisplaybreaks[4]
\usepackage{amsmath,amsthm,amscd,amssymb,amsfonts, amsbsy}
\usepackage{latexsym}
\usepackage{exscale}
\newtheorem{Theorem}[equation]{Theorem}
\newtheorem{Proposition}[equation]{Proposition}

\newtheorem{Lemma}[equation]{Lemma}

\def\XXint#1#2#3{{\setbox0=\hbox{$#1{#2#3}{\int}$}
\vcenter{\hbox{$#2#3$}}\kern-.5\wd0}}

\def\bbR{\mathbb{R}}

\def\ch{\rm{ch}}
\def\dw{\textup{d}w}

\newcommand{\abs}[1]{|#1|}
\newcommand{\ave}[1]{\langle #1 \rangle}

\begin{document}

\title[$A_p$-$A_\infty$ estimates for square functions]{Weak and strong $A_p$-$A_\infty$ estimates for square functions and related operators}

\author{Tuomas P. Hyt\"onen}

\address{Department of Mathematics and Statistics, P.O.B. 68 (Gustaf
H\"all\-str\"omin katu 2b), FI-00014 University of Helsinki, Finland}
\email{tuomas.hytonen@helsinki.fi}

\author{Kangwei Li}

\address{Department of Mathematics and Statistics, P.O.B. 68 (Gustaf
H\"all\-str\"omin katu 2b), FI-00014 University of Helsinki, Finland}
\curraddr{\sc{BCAM--Basque Center for Applied Mathematics, Mazarredo, 14. 48009 Bilbao, Basque Country, Spain}}

\email{kangwei.nku@gmail.com, \,\, kli@bcamath.org}

\thanks{T.P.H and K.L. are supported by the European Union through the ERC
Starting Grant
``Analytic-probabilistic methods for borderline singular integrals''.
They are members of the Finnish Centre of Excellence in Analysis and
Dynamics Research.
}

\date{\today}


\begin{abstract}
We prove sharp weak and strong type weighted estimates for a class of dyadic operators that includes majorants of both standard singular integrals and square functions. Our main new result is the optimal bound $[w]_{A_p}^{1/p}[w]_{A_\infty}^{1/2-1/p}\lesssim [w]_{A_p}^{1/2}$ for the weak type norm of square functions on $L^p(w)$ for $p>2$; previously, such a bound was only known with a logarithmic correction. By the same approach, we also recover several related results in a streamlined manner.
\end{abstract}

\maketitle

\section{Introduction}

We study weighted inequalities for the (in general nonlinear) operator
\[
A_{\mathcal S}^r (f)=\Big(\sum_{Q\in \mathcal S}\langle f \rangle_Q^r \mathbf 1_Q\Big)^{\frac 1r},
\qquad\langle f\rangle_Q:=\frac{1}{|Q|}\int_Q f,
\]
where $r>0$ and $\mathcal S$ is a sparse collection of dyadic cubes, i.e., there are pairwise disjoint subsets $E(S)\subset S$ such that $|E(S)|\geq\frac12|S|$. For $r=1$ and $r=2$, these operators dominate large classes of Calder\'on--Zygmund singular integrals and Littlewood--Paley square functions, respectively (see \cite{Lerner1,Lerner2} and \cite{L_A2} for details). Thus the various norm inequalities that we prove for $A_{\mathcal S}^r$ immediately translate to corresponding estimates for these classes of classical operators, recovering results like the $A_2$ theorem of the first author~\cite{Hytonen2012} and its several variants and elaborations.

More precisely, we are concerned with quantifying the dependence of various weighted operator norms on a mixture of the two-weight $A_p$ characteristic
\begin{equation*}
  [w,\sigma]_{A_p}:=\sup_Q\ave{w}_Q\ave{\sigma}_Q^{p-1}
\end{equation*}
and the individual $A_\infty$ characteristics
\begin{equation*}
  [w]_{A_\infty}:=\sup_Q\frac{1}{w(Q)}\int_Q M(1_Q w)
\end{equation*}
and $[\sigma]_{A_\infty}$. The study of such mixed bounds was initiated in \cite{HP}. All our estimates will be stated in a dual-weight formulation, in which the classical one-weight case corresponds to the choice $\sigma=w^{1-p'}$. Note that $[w,\sigma]_{A_p}$ becomes the usual one-weight $A_p$ characteristic $[w]_{A_p}:=[w,w^{1-p'}]_{A_p}$ with this choice.

Since we are dealing with dyadic operators, we also consider the dyadic versions of the weight characteristics, where the supremums above are over dyadic cubes only and $M$ denotes the dyadic maximal operator. This is a standing convention throughout this paper without further notice. Note, however, that the domination of classical operators typically involves a sum of boundedly many $A_{\mathcal S}^r$'s with respect to different dyadic systems, and for this reason the non-dyadic weight characteristics appear in such results.

The following strong type bound has been proved by Lacey and the second author in \cite{LL}, but we shall give a new proof here.

\begin{Theorem}\label{thm:m}
Let $1<p<\infty$ and $r>0$. Let $w,\sigma$ be a pair of weights. Then
\[
\|A_{\mathcal S}^r(\cdot \sigma)\|_{L^p(\sigma)\rightarrow L^p(w)}\le C [w,\sigma]_{A_p}^{\frac 1p}([w]_{A_\infty}^{(\frac 1r-\frac 1p)_+}+[\sigma]_{A_\infty}^{\frac 1p}).
\]
\end{Theorem}

Here and below, we simplify case analysis by interpreting $[w]_{A_\infty}^0=1$, whether or not $[w]_{A_\infty}$ is finite. The novelties of our approach are two-fold: we make black-box use of certain two-weight theorems, rather than adapting their proofs, and we avoid the ``slicing'' argument, namely, the separate consideration of families of cubes with the $A_p$ characteristic ``frozen'' to a certain value $\ave{w}_Q\ave{\sigma}_Q^{p-1}\eqsim 2^k\leq[w,\sigma]_{A_p}$.

For $r=1$, Theorem~\ref{thm:m} (in combination with the domination of singular integrals by $A_{\mathcal S}^1$) is the $A_p$-$A_\infty$ elaboration, by the first author and Lacey \cite{HL}, of the $A_2$ theorem of \cite{Hytonen2012}.  In this case, a ``slicing-free'' argument was provided in \cite{Hytonen}, but we feel that the present approach is simpler even here.

The benefits of this approach are best seen in the weak type estimate, for which we obtain the following new result:

\begin{Theorem}\label{thm:main}
Let $1<p<\infty$ with $p\neq r$. Let $w,\sigma$ be a pair of weights. Then
\[
\|A_{\mathcal S}^r(\cdot \sigma)\|_{L^p(\sigma)\rightarrow L^{p,\infty}(w)}\le C [w,\sigma]_{A_p}^{\frac 1p}[w]_{A_\infty}^{(\frac 1r-\frac 1p)_+}.
\]
\end{Theorem}

The case $p<r$ of Theorem~\ref{thm:main} was essentially known and due to Lacey and Scurry \cite{LS2012}, and we merely repeat their one-weight proof in the two-weight case. Note that we do not say anything about the critical exponent $p=r$, as our arguments do not shed any new light into this case.
For $p>r$, however, our bound
\begin{equation*}
  [w,\sigma]_{A_p}^{\frac 1p}[w]_{A_\infty}^{\frac 1r-\frac 1p}
  \lesssim[w]_{A_p}^{\frac 1p}[w]_{A_p}^{\frac 1r-\frac 1p}=[w]_{A_p}^{\frac 1r}
\end{equation*}
is new even in the one weight case $\sigma=w^{1-p'}$. Indeed, for $r=2$, the previous bounds in the literature had an additional logarithmic factor, taking the form $1+\log[w]_{A_p}$ in \cite{LS2012}, and subsequently improved to $(1+\log[w]_{A_\infty})^{\frac12}$ by Domingo-Salazar, Lacey, and Rey \cite{DLR}. By analogy to the failure of the $A_1$ conjecture (see \cite{NRVV}), a logarithmic correction is probably necessary in the critical case $p=r$. We are able to avoid it for $p>r$ by using a proof strategy specific to this range of exponents, whereas \cite{DLR,LS2012} treat all $p\geq r$ as one case.

Theorem~\ref{thm:main} with $r=2$ completes the picture of sharp weighted inequalities for square functions, aside from the remaining critical case of $p=2$. Namely, $[w]_{A_p}^{\max(\frac1p,\frac12)}$ is the optimal bound among all possible bounds of the form $\Phi([w]_{A_p})$ with an increasing function $\Phi$. This was shown by Lacey and Scurry \cite{LS2012} in the category of power type function $\Phi(t)=t^\alpha$; a variant of their argument proves the general claim, as we show in the last section.

%

To prove the above results, we need the following characterization, which is essentially due to Lai \cite{Lai}; we supply the necessary details to cover the cases that were not explicitly treated in \cite{Lai}.

\begin{Theorem}\label{thm:m1}
Let $1<p<\infty$ and $r>0$. Let $w,\sigma$ be a pair of weights. Then
\begin{equation*}
\begin{split}
\|A_{\mathcal S}^r(\cdot \sigma)\|_{L^p(\sigma)\rightarrow L^p(w)}^r &\simeq
\begin{cases}
\mathcal T+\mathcal T^*, & p>r,\\
\mathcal T, & 1<p\le r,
\end{cases} \\
\|A_{\mathcal S}^r(\cdot \sigma)\|_{L^p(\sigma)\rightarrow L^{p,\infty}(w)}^r &\simeq\ \ \mathcal T^*, \qquad p>r,
\end{split}
\end{equation*}
where
\begin{align*}
\mathcal T&= \sup_{R\in \mathcal S} \sigma(R)^{-\frac rp} \Big\| \sum_{\substack{Q\in\mathcal S\\ Q\subset R}} \langle \sigma\rangle_Q^r \mathbf 1_Q           \Big\|_{L^{\frac pr}(w)},\\
\mathcal T^*&= \sup_{R\in\mathcal S} w(R)^{-\frac 1{(\frac pr)'}} \Big\| \sum_{\substack{Q\in\mathcal S\\ Q\subset R}} \langle \sigma\rangle_Q^{r-1} \langle w\rangle_Q \mathbf 1_Q           \Big\|_{L^{(\frac pr)'}(\sigma)}.
\end{align*}
\end{Theorem}

The case $p>r$ of Theorems~\ref{thm:m} and \ref{thm:main} is a consequence of Theorem~\ref{thm:m1} and the following, which contains the technical core of this paper.

\begin{Proposition}\label{prop:test}
Let $r>0$ and $1<p<\infty$. For $\mathcal T$ and $\mathcal T^*$ as in Theorem~\ref{thm:m1}, we have
\begin{equation*}
  \mathcal T  \lesssim [w,\sigma]_{A_p}^{\frac rp}[\sigma]_{A_\infty}^{\frac rp},\\
\end{equation*}
and
\begin{equation*}
  \mathcal T^* \lesssim [w,\sigma]_{A_p}^{\frac rp}[w]_{A_\infty}^{1-\frac rp},\qquad p>r.
\end{equation*}
\end{Proposition}

The plan of the paper is as follows: We start with the proof of Theorem~\ref{thm:m1} and proceed to the estimation of the testing constant $\mathcal T$ and $\mathcal T^*$ as in Proposition~\ref{prop:test}. This completes the proof of Theorems~\ref{thm:m} and~\ref{thm:main} in the case of $p>r$.
The remaining case of Theorem~\ref{thm:main} for $p<r$ is then handled in Section~\ref{sec:weak}. In the final section, we discuss the sharpness of our weak type estimates by modifying the example given by Lacey and Scurry \cite{LS2012}.

\section{Proof of Theorem~\ref{thm:m1}}

As mentioned, Theorem~\ref{thm:m1} is essentially due to Lai \cite{Lai}. Here we make a slight change to extend the range of $r$ from $[1,\infty)$ to $(0,\infty)$. At the same time, we feel that our argument might be slightly easier, in that it makes no reference to the Rubio de Francia algorithm.

\subsection{The case $p>r$}
In this case, we first give the following lemma.
\begin{Lemma}\label{lm:hl} Let $w,\sigma$ be a pair of weights and $p>r>0$. Then
\begin{align}
\|A_{\mathcal S}^r (\cdot \sigma)\|_{L^p(\sigma)\rightarrow L^p(w)}^r &\simeq  \sup_{\|f\|_{L^p(\sigma)=1}}\Big\|\sum_{Q\in \mathcal S} \langle \sigma\rangle_Q^r\langle f^r \rangle_Q^\sigma\mathbf 1_Q\Big\|_{L^{\frac pr}(w)}\label{eq:char}\\
\|A_{\mathcal S}^r(\cdot \sigma)\|_{L^p(\sigma)\rightarrow L^{p,\infty}(w)}^r &\simeq \sup_{\|f\|_{L^p(\sigma)=1}}  \Big \| \sum_{Q\in \mathcal S} \langle \sigma\rangle_Q^r   \langle f^r\rangle_Q^\sigma    \mathbf 1_Q \Big\|_{L^{p/r,\infty}(w)}\label{eq:claim}.
\end{align}
\end{Lemma}
\begin{proof}
For convenience, denote by $Y^p(w)$ the target space $L^p(w)$ or $L^{p,\infty}(w)$. We have
\begin{align*}
&\|A_{\mathcal S}^r (\cdot \sigma)\|_{L^p(\sigma)\rightarrow Y^p(w)}^r
=\sup_{\|f\|_{L^p(\sigma)=1}}\Big\|\sum_{Q\in \mathcal S} \langle f\sigma\rangle_Q^r \mathbf 1_Q\Big\|_{Y^{\frac pr}(w)}\\
&= \sup_{\|f\|_{L^p(\sigma)=1}}\Big\|\sum_{Q\in \mathcal S} \langle \sigma\rangle_Q^r(\langle f \rangle_Q^\sigma)^r \mathbf 1_Q\Big\|_{Y^{\frac pr}(w)}\\
&\le  \sup_{\|f\|_{L^p(\sigma)=1}}\Big\|\sum_{Q\in \mathcal S} \langle \sigma\rangle_Q^r\langle(M_\sigma (f))^r \rangle_Q^\sigma\mathbf 1_Q\Big\|_{Y^{\frac pr}(w)}\\
&= \sup_{\|f\|_{L^p(\sigma)=1}}\Big\|\sum_{Q\in \mathcal S} \langle \sigma\rangle_Q^r\Big\langle\Big(\frac{M_\sigma (f)}{\|M_\sigma(f)\|_{L^p(\sigma)}}\Big)^r \Big\rangle_Q^\sigma\mathbf 1_Q\Big\|_{Y^{\frac pr}(w)} \|M_\sigma(f)\|_{L^p(\sigma)}^r\\
&\lesssim \sup_{\|g\|_{L^p(\sigma)=1}}\Big\|\sum_{Q\in \mathcal S} \langle \sigma\rangle_Q^r\langle g^r \rangle_Q^\sigma\mathbf 1_Q\Big\|_{Y^{\frac pr}(w)},
\end{align*}
where in the last step, we used the boundedness of $M_{\sigma}$ on $L^p(\sigma)$, and the bound is independent of $\sigma$. For the other direction, notice that
 \[
 \langle f^r\rangle_Q^\sigma \le \inf_{x\in Q} M_\sigma (f^r)(x)=(\inf_{x\in Q} M_{\sigma, r}(f)(x))^{r}\le (\langle M_{\sigma, r}(f)\rangle_Q^\sigma)^r,
 \]
 where $M_{\sigma,r}(f):=(M_\sigma(f^r))^{1/r}$.
With this observation, we have
\begin{align*}
&\sup_{\|f\|_{L^p(\sigma)=1}}\Big\|\sum_{Q\in \mathcal S} \langle \sigma\rangle_Q^r\langle f^r \rangle_Q^\sigma\mathbf 1_Q\Big\|_{Y^{\frac pr}(w)} \\
&\le \sup_{\|f\|_{L^p(\sigma)=1}}\Big\|\sum_{Q\in \mathcal S} \langle \sigma\rangle_Q^r(\langle M_{\sigma,r} f \rangle_Q^\sigma)^r\mathbf 1_Q\Big\|_{Y^{\frac pr}(w)}\\
&\le\sup_{\|f\|_{L^p(\sigma)=1}} \|A_{\mathcal S}^r (\cdot \sigma)\|_{L^p(\sigma)\rightarrow Y^p(w)}^r \|M_{\sigma,r} f\|_{L^p(\sigma)}^r\\
&\lesssim \|A_{\mathcal S}^r (\cdot \sigma)\|_{L^p(\sigma)\rightarrow Y^p(w)}^r,
\end{align*}
where in the last step, we use the boundedness of $M_{\sigma, r}$ on $L^p(\sigma)$ since $p>r$, and the bound is independent of $\sigma$.
This completes the proof of Lemma~\ref{lm:hl}.
\end{proof}

Now suppose that $C_1$ is the best constant such that
\[
\Big\|\sum_{Q\in \mathcal S} \langle \sigma\rangle_Q^r\langle f^r \rangle_Q^\sigma\mathbf 1_Q\Big\|_{Y^{\frac pr}(w)}\le C_1 \|f\|_{L^p(\sigma)}^r,
\]
i.e.,
\begin{equation}\label{eq:c1}
\Big\|\sum_{Q\in \mathcal S} \langle \sigma\rangle_Q^r\langle f \rangle_Q^\sigma\mathbf 1_Q\Big\|_{Y^{\frac pr}(w)}\le C_1 \|f\|_{L^{\frac pr}(\sigma)},
\end{equation}
Then
\[
\|A_{\mathcal S}^r (\cdot \sigma)\|_{L^p(\sigma)\rightarrow Y^p(w)} \simeq C_1^{\frac 1r}.
\]
Hence, we have reduced the problem to study \eqref{eq:c1}. We need the following result given by Lacey, Sawyer and Uriarte-Tuero \cite{LSU}.
\begin{Proposition}\label{prop:char}
Let $\tau=\{\tau_Q: Q\in\mathcal Q\}$ be non-negative constants, $w,\sigma$ be weights and $T$ is the linear operator defined by
\[
T_\tau(f):= \sum_{Q\in\mathcal Q} \tau_Q \langle f\rangle_Q \mathbf 1_Q.
\]
Then for $1<p<\infty$, there holds
\begin{eqnarray*}
\|T_\tau (\cdot \sigma) \|_{L^p(\sigma)\rightarrow L^{p,\infty}(w)}& \simeq& \sup_{R\in\mathcal Q} w(R)^{-\frac 1{p'}} \Big\| \sum_{\substack{Q\in\mathcal Q\\ Q\subset R}} \tau_Q \langle w\rangle_Q \mathbf 1_Q           \Big\|_{L^{p'}(\sigma)}\\
\|T_\tau (\cdot \sigma) \|_{L^p(\sigma)\rightarrow L^p(w)}& \simeq& \sup_{R\in\mathcal Q} w(R)^{-\frac 1{p'}} \Big\| \sum_{\substack{Q\in\mathcal Q\\ Q\subset R}} \tau_Q \langle w\rangle_Q \mathbf 1_Q           \Big\|_{L^{p'}(\sigma)}\\
&&\quad+\sup_{R\in \mathcal Q} \sigma(R)^{-\frac 1p} \Big\| \sum_{\substack{Q\in\mathcal Q\\  Q\subset R}} \tau_Q \langle \sigma\rangle_Q \mathbf 1_Q           \Big\|_{L^p(w)}.
\end{eqnarray*}
\end{Proposition}

Observing that
\begin{equation*}
  LHS\eqref{eq:c1}
  =\|T_\tau(f\sigma)\|_{Y^{\frac pr}(w)}
\end{equation*}
with $\tau_Q=\ave{\sigma}_Q^{r-1}$,
Theorem~\ref{thm:m1} follows immediately from Proposition~\ref{prop:char}.


\subsection{The case $1<p\le r$} In this case, making use of the usual construction principal cubes $\mathcal F$ of $(f,\sigma)$, we have
\begin{align*}
\|A_{\mathcal S}^r (f\sigma)\|_{L^p(\sigma)\rightarrow L^p(w)}&= \Big\|  \Big( \sum_{Q\in\mathcal S} \langle f\sigma\rangle_Q^r        \mathbf 1_Q    \Big)^{\frac 1r}  \Big \|_{L^p(w)}\\
&\lesssim  \Big\|  \Big(\sum_{F\in \mathcal F}(\langle f\rangle_F^\sigma)^r \sum_{\substack{Q\in\mathcal S\\ \pi(Q)=F}} \langle \sigma\rangle_Q^r        \mathbf 1_Q    \Big)^{\frac 1r}  \Big \|_{L^p(w)}\\
&\le  \Big(    \sum_{F\in \mathcal F}(\langle f\rangle_F^\sigma)^p\Big\|  \Big(\sum_{\substack{Q\in\mathcal S\\ \pi(Q)=F}} \langle \sigma\rangle_Q^r        \mathbf 1_Q    \Big)^{\frac 1r}     \Big\|_{L^p(w)}^p                          \Big)^{\frac 1p}\\
&\le   \Big(    \sum_{F\in \mathcal F}(\langle f\rangle_F^\sigma)^p \mathcal T^{\frac pr }\sigma (F)     \Big)^{\frac 1p}
\lesssim \mathcal T^{\frac 1r}\|f\|_{L^p(\sigma)}
\end{align*}
On the other hand, it is obvious that
\[
\mathcal T^{\frac 1r}\le \|A_{\mathcal S}^r (\cdot \sigma)\|_{L^p(\sigma)\rightarrow L^p(w)}.
\]
Therefore, $\|A_{\mathcal S}^r (\cdot \sigma)\|_{L^p(\sigma)\rightarrow L^p(w)}\simeq \mathcal T^{\frac 1r}$.

\section{Proof of Proposition~\ref{prop:test}}

We recall the following proposition.

\begin{Proposition}[\cite{cov2004}, Proposition~2.2]\label{prop:dyadicSum}
Let $1<s<\infty$, $\sigma$ be a positive Borel measure and
\[
\phi=\sum_{Q\in\mathcal D} \alpha_Q \mathbf 1_Q,\qquad \phi_Q=\sum_{Q'\subset Q}\alpha_{Q'} \mathbf 1_{Q'}.
\]
Then
\[
\|\phi\|_{L^s(\sigma)}\eqsim \Big( \sum_{Q\in \mathcal D} \alpha_Q (\langle\phi_Q\rangle_Q^\sigma)^{s-1}\sigma(Q) \Big)^{1/s}.
\]
\end{Proposition}

We also need the following result, whose proof is based on the Kolmogorov's inequality.

\begin{Proposition}\cite[Lemma 5.2]{Hytonen}
Let $\gamma \in [0, 1)$. Then $$\sum_{\substack{Q\in \mathcal S\\ Q\subset R}} \langle w\rangle_Q^\gamma |Q|\lesssim \langle w\rangle_R^\gamma |R|.$$
\end{Proposition}

Now we can estimate the two testing constants.

\subsection{Estimate of $\mathcal T$}

Let us first note that the case $p\ge r+1$ implies the general case. Indeed, suppose the mentioned case is already proven, and consider $p< r+1$. Let $\mathcal T_r$ denote the testing constant related to a given value of $r$. Now in particular $r>p-1$, and hence
\begin{equation*}
\begin{split}
  \mathcal T_r^{1/r}
  &=\sup_{R\in\mathcal S}\sigma(R)^{-\frac 1p}\Big\|\Big(\sum_{\substack{Q\in\mathcal S\\ Q\subset R}}\ave{\sigma}_Q^r 1_Q\Big)^{\frac 1r}\Big\|_{L^p(w)}  \\
  &\leq\sup_{R\in\mathcal S}\sigma(R)^{-\frac 1p}\Big\|\Big(\sum_{\substack{Q\in\mathcal S\\ Q\subset R}}\ave{\sigma}_Q^{p-1} 1_Q\Big)^{\frac 1{p-1}}\Big\|_{L^p(w)}
  =  (\mathcal T_{p-1})^{\frac 1{p-1}}.
\end{split}
\end{equation*}
Since obviously $p\ge (p-1)+1$, we know by assumption that $\mathcal (\mathcal T_{p-1})^{\frac 1{p-1}}\leq[w,\sigma]_{A_p}^{\frac 1p}[\sigma]_{A_\infty}^{\frac 1p}$, and this gives the required bound for $\mathcal T_r$.

So we concentrate on $p\ge r+1$. By Proposition~\ref{prop:dyadicSum},
we have
\begin{align*}
&\Big\| \sum_{\substack{Q\in\mathcal S\\ Q\subset R}} \langle \sigma\rangle_Q^r \mathbf 1_Q           \Big\|_{L^{\frac pr}(w)} \\
  &\eqsim \Big(\sum_{\substack{Q\in \mathcal S\\ Q\subset R}}  \langle \sigma\rangle_Q^r w(Q) \Big(\frac 1{w(Q)} \sum_{\substack{Q'\in \mathcal S \\Q'\subset Q}}      \langle \sigma\rangle_{Q'}^r w(Q')                                   \Big)^{\frac pr -1}        \Big)^{\frac rp}\\
&=\Big(\sum_{\substack{Q\in \mathcal S\\ Q\subset R}}  \langle \sigma\rangle_Q^r w(Q) \Big(\frac 1{w(Q)} \sum_{\substack{Q'\in \mathcal S \\Q'\subset Q}}      \langle \sigma\rangle_{Q'}^r \langle w \rangle_{Q'}^{\frac r{p-1}}       \langle w \rangle_{Q'}^{1-\frac r{p-1}}             |Q'|       \Big)^{\frac pr -1}        \Big)^{\frac rp}\\
&\le [w,\sigma]_{A_p}^{\frac r{p-1}(1-\frac rp)}\Big(\sum_{\substack{Q\in \mathcal S\\ Q\subset R}}  \langle \sigma\rangle_Q^r w(Q) \Big(\frac 1{w(Q)} \sum_{\substack{Q'\in \mathcal S \\Q'\subset Q}}       \langle w \rangle_{Q'}^{1-\frac r{p-1}}             |Q'|       \Big)^{\frac pr -1}        \Big)^{\frac rp}\\
&\lesssim [w,\sigma]_{A_p}^{\frac r{p-1}(1-\frac rp)}\Big(\sum_{\substack{Q\in \mathcal S\\ Q\subset R}}  \langle \sigma\rangle_Q^r w(Q) \Big(\frac 1{w(Q)}       \langle w \rangle_{Q}^{1-\frac r{p-1}}             |Q|       \Big)^{\frac pr -1}        \Big)^{\frac rp}\\
&= [w,\sigma]_{A_p}^{\frac r{p-1}(1-\frac rp)}\Big(\sum_{\substack{Q\in \mathcal S\\ Q\subset R}}  \langle \sigma\rangle_Q^r \langle w\rangle_Q^{\frac{r-1}{p-1}} |Q|        \Big)^{\frac rp}\\
&\le [w,\sigma]_{A_p}^{\frac r{p-1}(1-\frac rp)+\frac{r-1}{p-1}\cdot\frac rp}\Big( \sum_{\substack{Q\in \mathcal S\\ Q\subset R}}  \langle \sigma\rangle_Q |Q|\Big)^{\frac rp}
\lesssim [w,\sigma]_{A_p}^{\frac rp} [\sigma]_{A_\infty}^{\frac rp} \sigma(R)^{\frac rp}.
\end{align*}
Therefore,
\begin{equation}\label{eq:1}
\mathcal T\lesssim [w,\sigma]_{A_p}^{\frac rp} [\sigma]_{A_\infty}^{\frac rp}.
\end{equation}

\subsection{Estimate of $\mathcal T^*$}

Recall that we only consider $p>r$ in this case. For simplicity, we denote $s=(\frac pr)'$. By Proposition~\ref{prop:dyadicSum} again, we have
\begin{equation}\label{eq:s}
\begin{split}
&\Big\| \sum_{\substack{Q\in\mathcal S\\ Q\subset R}} \langle \sigma\rangle_Q^{r-1} \langle w\rangle_Q \mathbf 1_Q           \Big\|_{L^{(\frac pr)'}(\sigma)}\\
&\eqsim \Big( \sum_{\substack{Q\in\mathcal S\\ Q\subset R}}  \langle \sigma\rangle_Q^{r-1} \langle w\rangle_Q\Big(\frac 1{\sigma(Q)} \sum_{\substack{Q'\in \mathcal S \\Q'\subset Q}} \langle \sigma\rangle_{Q'}^{r-1} \langle w\rangle_{Q'} \sigma(Q')\Big)^{s-1}\sigma(Q) \Big)^{1/s}
\end{split}
\end{equation}
We consider $r<p<r+1$ and $p>r+1$ separately. If $r<p<r+1$, then
\begin{align*}
&RHS \eqref{eq:s}\\
&= \Big( \sum_{\substack{Q\in\mathcal S\\ Q\subset R}}  \langle \sigma\rangle_Q^{r-1} \langle w\rangle_Q\Big(\frac 1{\sigma(Q)} \sum_{\substack{Q'\in \mathcal S \\Q'\subset Q}} \langle \sigma\rangle_{Q'}^{p-1} \langle w\rangle_{Q'} \langle\sigma\rangle_{Q'}^{r+1-p} |Q'|\Big)^{s-1}\sigma(Q) \Big)^{1/s}\\
&\le [w,\sigma]_{A_p}^{\frac {s-1}s} \Big( \sum_{\substack{Q\in\mathcal S\\ Q\subset R}}  \langle \sigma\rangle_Q^{r-1} \langle w\rangle_Q\Big(\frac 1{\sigma(Q)} \sum_{\substack{Q'\in \mathcal S \\Q'\subset Q}} \langle\sigma\rangle_{Q'}^{r+1-p} |Q'|\Big)^{s-1}\sigma(Q) \Big)^{1/s}\\
&\lesssim [w,\sigma]_{A_p}^{\frac {s-1}s} \Big( \sum_{\substack{Q\in\mathcal S\\ Q\subset R}}  \langle \sigma\rangle_Q^{r-1} \langle w\rangle_Q\Big(\frac 1{\sigma(Q)}  \langle\sigma\rangle_{Q}^{r+1-p} |Q|\Big)^{s-1}\sigma(Q) \Big)^{1/s}\\
&= [w,\sigma]_{A_p}^{\frac rp}\Big(\sum_{\substack{Q\in\mathcal S\\ Q\subset R}} \langle w\rangle_Q |Q|\Big)^{1/s}
\lesssim [w,\sigma]_{A_p}^{\frac rp}[w]_{A_\infty}^{1-\frac rp}w(R)^{1/(\frac pr)'}.
\end{align*}
If $p\ge r+1$, then
\begin{align*}
&RHS \eqref{eq:s}\\
&= \Big( \sum_{\substack{Q\in\mathcal S\\ Q\subset R}}  \langle \sigma\rangle_Q^{r-1} \langle w\rangle_Q\Big(\frac 1{\sigma(Q)} \sum_{\substack{Q'\in \mathcal S \\Q'\subset Q}} \langle \sigma\rangle_{Q'}^{r} \langle w\rangle_{Q'}^{\frac {r}{p-1}} \langle w\rangle_{Q'}^{1-\frac r{p-1}}|Q'|\Big)^{s-1}\sigma(Q) \Big)^{1/s}\\
&\le [w,\sigma]_{A_p}^{\frac {r^2}{(p-1)p}} \Big( \sum_{\substack{Q\in\mathcal S\\ Q\subset R}}  \langle \sigma\rangle_Q^{r-1} \langle w\rangle_Q\Big(\frac 1{\sigma(Q)} \sum_{\substack{Q'\in \mathcal S \\Q'\subset Q}}  \langle w\rangle_{Q'}^{1-\frac r{p-1}}|Q'|\Big)^{s-1}\sigma(Q) \Big)^{1/s}\\
&\le  [w,\sigma]_{A_p}^{\frac {r^2}{(p-1)p}} \Big( \sum_{\substack{Q\in\mathcal S\\ Q\subset R}}  \langle \sigma\rangle_Q^{r-1} \langle w\rangle_Q\Big(\frac 1{\sigma(Q)}  \langle w\rangle_{Q}^{1-\frac r{p-1}}|Q|\Big)^{s-1}\sigma(Q) \Big)^{1/s}\\
&= [w,\sigma]_{A_p}^{\frac {r^2}{(p-1)p}}\Big(\sum_{\substack{Q\in\mathcal S\\ Q\subset R}} \langle w \rangle_Q^{1+ \frac {(p-1-r)r}{(p-1)(p-r)}} \langle\sigma \rangle_Q^{\frac {(p-1-r)r}{p-r}}|Q|\Big)^{1/s}\\
&\le [w,\sigma]_{A_p}^{\frac {r^2}{(p-1)p}+\frac{(p-1-r)r}{p(p-1)}} \Big(\sum_{\substack{Q\in\mathcal S\\ Q\subset R}} \langle w \rangle_Q   |Q|\Big)^{1/s}
\lesssim [w,\sigma]_{A_p}^{\frac rp} [w]_{A_\infty}^{1-\frac rp} w(R)^{1/(\frac pr)'}.
\end{align*}
Therefore, in both cases,
\begin{equation}\label{eq:2}
\mathcal T^*\lesssim [w,\sigma]_{A_p}^{\frac rp} [w]_{A_\infty}^{1-\frac rp} .
\end{equation}

Combining  \eqref{eq:1} and \eqref{eq:2}, we have completed the proof of Proposition~\ref{prop:test}. Together with Theorem~\ref{thm:m1}, this yields Theorem~\ref{thm:m} as well as Theorem~\ref{thm:main} in the case that $p>r$.


\section{Proof of the weak type bound for $1<p<r$}\label{sec:weak}

We are left to prove Theorem~\ref{thm:main} in the case that $1<p<r$. Actually, Lacey and Scurry \cite{LS2012} have investigated the one-weight case. Following their method, it is easy to give the two-weight estimate as well. For completeness, we give the details. We want to bound the following inequality,
\[
\sup_{t>0} t w(\{x\in \bbR^n: A_{\mathcal S}^r (f\sigma)>t \})^{\frac 1p}\le C \|f\|_{L^p(\sigma)}.
\]
By scaling it suffices to give an uniform estimate for
\[
t_0 w(\{x\in \bbR^n: A_{\mathcal S}^r (f\sigma)>t_0 \})^{\frac 1p},
\]where $t_0$ is some constant to be determined later.
It is also free to further sparsify $\mathcal S$ such that
\[
\Big|\bigcup_{\substack{Q'\subsetneq Q\\ Q',Q\in \mathcal S} }Q'\Big| \le \frac 14 |Q|.
\]
Now set
\[
\mathcal S_l:=\{Q\in \mathcal S: 2^{-l-1 }< \langle f\sigma\rangle_Q\le 2^{-l} \},\qquad l\ge 0,
\]  and
\[
\mathcal S_{-1}:=\{Q\in \mathcal S: \langle f\sigma\rangle_Q> 1\}.
\]
Then for $Q\in \mathcal S_l$, $l\ge 0$, denote by $\ch_{\mathcal S_l}(Q)$ the maximal subcubes of $Q$ in $\mathcal S_l$ and $E_Q=Q\setminus (\cup_{Q'\in
\ch_{\mathcal S_l}(Q)} Q')$. We have
\begin{align*}
\langle f\sigma \mathbf 1_{E_Q}\rangle_Q &= \frac 1{|Q|} \int_{Q} f\sigma dx -\frac 1{|Q|}\sum_{   Q'\in \ch_{\mathcal S_l}(Q) } \int_{Q'} f\sigma dx\\
&= \frac 1{|Q|} \int_{Q} f\sigma dx -\sum_{  Q'\in \ch_{\mathcal S_l}(Q) }\frac {|Q'|}{|Q|} \frac {1}{|Q'|} \int_{Q'} f\sigma dx\\
&\ge \frac 1{|Q|} \int_{Q} f\sigma dx-\frac 14 2^{-l}
\ge  \frac 12 \langle f\sigma\rangle_Q.
\end{align*}
Since
\begin{align*}
&w(\{x\in \bbR^n: A_{\mathcal S}^r (f\sigma)>t_0 \}) \\
&= w(\{x\in \bbR^n:  \sum_{Q\in \mathcal S} \langle f\sigma\rangle_Q^r \mathbf 1_Q>t_0^r \}) \\
&\le  w(\{x\in \bbR^n:  \sum_{l\ge 0}\sum_{Q\in \mathcal S_l} \langle f\sigma\rangle_Q^r \mathbf 1_Q>\frac {t_0^r}2  \}) \\
&\qquad + w(\{x\in \bbR^n:  \sum_{Q\in \mathcal S_{-1}} \langle f\sigma\rangle_Q^r \mathbf 1_Q>\frac {t_0^r}2   \})
=:I_1+I_2.
\end{align*}
It is easy to see that
\begin{align*}
I_2\le w(\cup_{S\in \mathcal S_{-1}} S)\le w(\{M(f\sigma) > 1\})\lesssim [w,\sigma]_{A_p} \|f\|_{L^p(\sigma)}^p.
\end{align*}
Let $\frac {t_0^r}2= \sum_{l\ge 0} 2^{-\epsilon l} $, where $\epsilon:= (r-p)/2$. We have
\begin{align*}
I_1&\le \sum_{l\ge 0} w(\{x\in \bbR^n:\sum_{Q\in \mathcal S_l} \langle f\sigma\rangle_Q^r \mathbf 1_Q> 2^{-\epsilon l} \})\\
&\le  \sum_{l\ge 0 }w(\{x\in \bbR^n:\sum_{Q\in \mathcal S_l} \langle f\sigma\rangle_Q^p   \mathbf 1_Q> 2^{(r-p)l}2^{-\epsilon l} \})\\
&\le \sum_{l\ge 0 }w(\{x\in \bbR^n:\sum_{Q\in \mathcal S_l} \langle f\sigma\mathbf 1_{E_Q}\rangle_Q^p   \mathbf 1_Q> 2^{-p}2^{(r-p)l}2^{-\epsilon l} \})\\
&\le \sum_{l\ge 0 } 2^{(p+\epsilon-r)l+p}\int_{\bbR^n } \sum_{Q\in \mathcal S_l} \langle f\sigma\mathbf 1_{E_Q}\rangle_Q^p   \mathbf 1_Q \dw\\
&\le  \sum_{l\ge 0 } 2^{(p+\epsilon-r)l+p} \sum_{Q\in \mathcal S_l}\frac {w(Q)}{|Q|^p} \sigma(Q)^{p-1}\int_{E_Q} f^p \textup{d}\sigma \\
&\lesssim [w,\sigma]_{A_p} \|f\|_{L^p(\sigma)}^p.
\end{align*}
Combining the above, we get
\[
\|A_{\mathcal S}^r (f\sigma)\|_{L^{p,\infty}(w)}\lesssim [w,\sigma]_{A_p}^{\frac 1p} \|f\|_{L^p(\sigma)}.
\]

\section{Sharpness of the weak type bounds}

In this section, let
\begin{equation*}
  Sf:=\Big(\sum_{I\in\mathcal D}\frac{1_I}{\abs{I}}\abs{\langle h_I,f\rangle}^2\Big)^{1/2}
\end{equation*}
denote the Haar square function, and $\sigma:=w^{1-p'}$ will always be the $L^p$-dual weight of $w$ for a fixed value of $p\in(1,\infty)$. We show that the norm bound $\|S\|_{L^p(w)\to L^{p,\infty}(w)}\lesssim[w]_{A_p}^{\max(\frac1p,\frac12)}$ is unimprovable. Actually, a lower bound with the exponent $\frac1p$ holds uniformly over all weights, which is the content of the next (straightforward) proposition. The optimality of the exponent $\frac12$ is slightly more tricky, and is based on a (standard) example of a specific weight.

\begin{Proposition}
For any weight $w$, we have
\begin{equation*}
  \|S \|_{L^p(w)\rightarrow L^{p,\infty}(w)}\ge [w]_{A_p}^{\frac 1p}.
\end{equation*}
\end{Proposition}

\begin{proof}
Let $N:=\|S \|_{L^p(w)\rightarrow L^{p,\infty}(w)}$ and consider
 $f=\operatorname{sgn}(h_I)\abs{f}$. Then $Sf\geq 1_I\abs{I}^{-1/2}\langle \abs{h_I},\abs{f}\rangle=1_I\ave{\abs{f}}_I$. Thus
\begin{equation*}
  N\|f\|_{L^p(w)}
  \geq \|1_I\ave{\abs{f}}_I\|_{L^{p,\infty}(w)}
  =\frac{w(I)^{1/p}}{\abs{I}} \int_I\abs{f}= \frac{w(I)^{1/p}}{\abs{I}} \int_I\abs{f} w^{-1}w
\end{equation*}
for all positive functions $\abs{f}$ on $I$. By the converse to H\"older's inequality, this shows that
\begin{equation*}
  N\geq\frac{w(I)^{1/p}}{\abs{I}}\|w^{-1}\|_{L^{p'}(w)}=\frac{w(I)^{1/p}\sigma(I)^{1/p'}}{\abs{I}},
\end{equation*}
and taking the supremum over all $I$ proves the claim.
\end{proof}

\begin{Proposition}
Let $\Phi$ be an increasing function such that
\begin{equation*}
  \|S \|_{L^p(w)\rightarrow L^{p,\infty}(w)}\le \Phi([w]_{A_p})
\end{equation*}
for all $w\in A_p$. Then $\Phi(t)\geq c t^{1/2}$.
\end{Proposition}

Lacey and Scurry \cite{LS2012} showed that this in the class of power functions, namely, they proved that there cannot be a bound of the form $\Phi(t)=t^{1/2-\eta}$ for $\eta>0$. The stronger claim above follows by an elaboration of their argument.

\begin{proof}
Following the same arguments as that in \cite{LS2012}, the assumption implies
\[
\Big\|  \Big( \sum_Q \langle a_Q\cdot w\rangle_Q^2 \mathbf 1_Q   \Big)^{1/2}   \Big\|_{L^{p'}(\sigma)}\lesssim \Phi([w]_{A_p}) \Big\| (\sum_Q a_Q^2)^{1/2}  \Big\|_{L^{p',1}(w)}
\]
for all sequences of measurable functions $a_Q$.
For $\varepsilon>0$, we consider $w(x)=\abs{x}^{\varepsilon-1}$ and a sequence of functions
\[
   a_{[0,2^{-k})}(x):=a_k(x):= \varepsilon^{\frac 12}\sum_{j=k+1}^\infty 2^{-\varepsilon(j-k)} \mathbf 1_{[2^{-j}, 2^{-j+1})}(x),\quad k\in \mathbb N.
\]
Then it is easy to check that $[w]_{A_p}\simeq w([0,1])\simeq \varepsilon^{-1}$ and
 $\sum_k a_k(x)^2 \lesssim \mathbf 1_{[0,1]}$ so that
\[
\Big\| (\sum_{k=1}^\infty a_k(x)^2)^{1/2}  \Big\|_{L^{p',1}(w)}\lesssim w([0,1])^{1/{p'}}.
\]
On the other hand,
\[
\langle a_k \cdot w\rangle_{[0, 2^{-k})}\simeq \varepsilon^{\frac 12} 2^k \sum_{j=k+1}^\infty 2^{-\varepsilon (j-k)} 2^{-\varepsilon j}\simeq \varepsilon^{-\frac 12} 2^{k(1-\varepsilon)}.
\]
It follows that
\[
   \int_{[0,1]}\Big( \sum_{k=1}^\infty \langle a_k\cdot w\rangle_{[0, 2^{-k})}^2 \mathbf 1_{[0,2^{-k})}   \Big)^{p'/2}  \textup{d} \sigma
   \simeq\varepsilon^{-p'/2-1} \simeq \varepsilon^{-p'/2} w([0,1]).
\]
By assumption, this implies $\varepsilon^{-1/2}\lesssim\Phi([w]_{A_p})\leq\Phi(c\varepsilon^{-1})$, and hence $\Phi(t)\gtrsim t^{1/2}$.
\end{proof}


\end{document}